\documentclass[12 pt, a4paper]{amsart}
\usepackage{amssymb,latexsym,epsfig,color}
\usepackage{enumerate}
\usepackage[all,cmtip]{xy}
\usepackage[utf8]{inputenc}
\address{
}


\usepackage{color}

\newcommand{\N}{{\mathbb N}}

\newcommand{\Q}{{\mathbb Q}}
\newcommand{\br}{{\mathbf R}}
\newcommand{\bt}{{\mathbf T}}
\newcommand{\lone}{{ L_1}}
\newcommand{\ltwo}{{ L_2}}
\newcommand{\lthree}{{ L_3}}
\newcommand{\li}{{ L_i}}
\newcommand{\lfour}{{ L_4}}

\newcommand{\bb}{{\mathbf B}}

\newcommand{\ba}{{\mathbf A}}
\newcommand{\bs}{{\mathbf {Set}}}

\newcommand{\RR}{{\mathcal R}}

\newcommand{\Set}{{\mathcal S}}

\newcommand{\PP}{{\mathcal P}}

\newtheorem{theorem}{Theorem}

\newtheorem{corollary}[theorem]{Corollary}
\newtheorem{lemma}[theorem]{Lemma}

\theoremstyle{definition}

\newtheorem{remark}[theorem]{Remark}
\newtheorem{example}[theorem]{Example}

\title{Betweenness relations in a categorical setting
 }
\author{J. Bruno, A. McCluskey and P. Szeptycki}
\begin{document}
\maketitle
\begin{abstract} We apply a categorical lens to the study of betweenness relations by capturing them within a topological category, fibred in lattices, and study several subcategories of it. In particular, we show that its full subcategory of finite objects forms a Fraiss\'{e} class implying the existence of a countable homogenous betweenness relation.
We furthermore show that the subcategory of antisymmetric betweenness relations is reflective. As an application we recover the reflectivity of distributive complete lattices within complete lattices, and we end with some observations on the Dedekind-MacNeille completion.
\end{abstract}

\section{Introduction and background}

The study of betweenness relations dates back as far the late 1800's (\cite{OldBeet}) with sporadic revivals throughout the last century whose focus lie in characterising certain partial orders in terms of the betweenness relations they generate (see for example \cite{ResultsChara}, \cite{MR0313134},
\cite{MR1120620}, \cite{MR1501071}, \cite{MR2785544}, \cite{MR1383662} and  \cite{MR1362944}). An excellent modern approach can be found in \cite{pre06252969} where Bankston explores a large variety of settings in which betweenness relations arise, with emphasis on ordered sets, metric spaces and continuum topology. The author makes a strong case for considering the simple notion of a {\it road system} on an arbitrary set as a natural approach to generating the intuitive idea of betweenness relations as ternary relations. In the broadest of interpretations, for a point $b$ to lie {\it between} points $a$ and $c$ it must be that any way to get from $a$ to $c$ must inevitably go though $b$. Intuitively, given an arbitrary point, it is desirable that no other point lies between it and itself. If, in addition, one demands for a way to reach any point from any other, then one is in the presence of a {\bf road system} (in the sense of Bankston): a set $X$ with a collection of {\bf roads}, $\RR \subseteq \PP(X)$, so that (a) for any point $x\in X$, $\{x\} \in \RR$ and (b) for any pair $\{x,y\} \subseteq X$, $\{x,y\} \subseteq R \in \RR$ for some $R$. In this sense, a point lies between a pair of points provided that it is contained in any road containing the pair. Examples of road systems abound: branches in a tree, (connected)-components (resp. continuum connected) in connected topological spaces (resp. continuum-wise connected), intervals in order sets, linear subspaces in real vector spaces, etc. A natural setting for betweenness of points can be found in ternary relations; for an arbitrary set $X$, a ternary relation $R\subseteq X^3$ can be interpreted as a betweenness relation by reading $(a,b,c)$ as $b$ lies between $a$ and $c$. In the language of Bankston, a ternary relation arising from a road system is called an {\bf R-relation} and it turns out that this primitive notion of betweenness can be captured in a very basic first-order language with only one ternary predicate symbol (i.e., $(\cdot,\cdot,\cdot)$) and equality. 

From a categorical view point, the above work is concerned with concrete objects in a suitable chosen category. The canonical choice of arrows for R-relations would be those where {\it betweenness} is preserved (i.e., monotone functions between ternary relations). It is this choice of morphisms and the category, $\br$, it generates that outlines the scope of this paper. More precisely, we study dualities between several types of ordered sets and the R-relations they generate. As an application, we illustrate a duality between distributive lattices and the {\it antisymmetric} R-relations developed from such and, thereby, establish a reflector between a category of complete lattices and its full subcategory of their distributive counterparts (i.e., the antisymmetric closure of an R-relation). The layout for this paper is the following:  in Section~\ref{sec:r-relations} we begin the study of R-relations by interpreting its defining first-order axioms as categories and adjunctions thereof. Here we also show $\br$ is a topological category and, as a consequence, that its full subcategory of finite objects forms a Fraiss\'{e} class, giving rise to a homogeneous relational R-relation in $\br$ (see \cite{MR2800979}). Exploiting Grothendieck's construction, we also illustrate the forgetful functor $\br \to \bs$ as a Grothendieck fibration.
Section~\ref{sec:sepanddist} begins by locating the category of antisymmetric R-relations as reflective subcategory of its ambient category. We construct this reflector as an inverse limit of length $\omega$. Lastly, we make use of this adjunction, in addition to results from previous sections, to establish dualities between certain types of ordered sets and the R-relations they generate; in particular, the aforementioned duality between distributive lattices and antisymmetric R-relations. We end this section by highlighting some minor observations to do with the Dedekind-MacNeille completion of a partial order from a betweenness perspective.

For the most part this article is self contained and notation is standard. However, we make use of the following: ternary relations will be denoted by $[\cdot,\cdot,\cdot]$ with subscripts to differentiate between them (i.e. $[\cdot,\cdot,\cdot]_X$, $[\cdot,\cdot,\cdot]_\RR$, etc). For a set $X$ and a relation $[\cdot,\cdot,\cdot]$ on it, we denote $(a,b,c) \in [\cdot,\cdot,\cdot]$ by $[a,b,c]$; subscripts are also included in this notation (e.g. $(a,b,c) \in [\cdot,\cdot,\cdot]_\RR$, is denoted as $[a,b,c]_\RR$). We also remark that throughout the paper a {\it monotone} function refers to an order-preserving function between ternary relations and while order-preservation between ordered sets is simply referred to as {\it ordered-preserving} functions. All categorical notions are standard and can be found in \cite{MR1712872} and \cite{MR2240597}.

\section{R-relations}\label{sec:r-relations}
Let $\bt$ denote the bicomplete category whose objects are sets endowed a with ternary relation and whose morphisms are monotone functions (i.e., for objects  $(X,[\cdot,\cdot,\cdot]_X)$ and $(Y,[\cdot,\cdot,\cdot]_Y)$ a function $f: X \to Y$ is a morphism provided $[a,b,c]_X \Rightarrow [f(a),f(b),f(c)]_Y$. The forgetful functor $\bt \to \bs$ is left and right adjoint and, therefore, the underlying sets for limits and colimits in $\bt$ are those of limits and colimits in $\bs$.

 For a set $X$, a {\bf road system} $\RR$ on it is a collection of subsets of $X$ so that: (a) any singleton is in $\RR$ and (b) for any two points in $X$ there exists a set in $\RR$ containing both. Any such road system generates a ternary relation $[\cdot,\cdot,\cdot]_\RR$ on $X$ called an {\bf R-relation} where $[a,b,c]_\RR$ if, and only if, $b\in R$ for any $R \in \RR$ with $a,c\in R$. Using Bankston's notation, we let 

$$[a,b]:=\{c\in X \mid [a,c,b]_\RR\} = \bigcap_{a,b\in R}R.$$

 In \cite{pre06252969} the author illustrates how ternary relations arising from road systems are first-order axiomatizable. 

\begin{lemma}[Bankston] A relation $[\cdot,\cdot,\cdot]$ on a set $X$ can be generated from a road system if, and only if, the following universal axioms hold.

\begin{itemize}
\item[ ] (R1)  Reflexivity: $[a,b,b]$;
\medskip
\item[ ] (R2) Symmetry: $[a,b,c] \Rightarrow [c,b,a]$;
\medskip
\item[ ] (R3) Minimality: $[a,b,a] \Rightarrow a=b$; and
\medskip
\item[ ] (R4) Transitivity: $[a,b,c] \land [a,d,c] \land [b,x,d] \Rightarrow [a,x,c]$.
\end{itemize}
 
\end{lemma}

\noindent
For instance, (R4) reads as 
$$\forall a,b,c,d,x\in X\left([a,b,c], [a,d,c], [b,x,d] \Rightarrow  [a,x,c] \right).$$
\noindent
An R$_i$-relation, $i\leq 4$, is a ternary relation satisfying its corresponding universal axiom. For any set $X$, the smallest ternary relation on it that qualifies as an R-relation is $X_\bot:=\{(a,b,b),(b,b,a)\mid a,b\in X\}$,  while the largest is $X_\top:=X^3 \smallsetminus \{(a,b,a)\mid a\not = b\}$. The full subcategory $\br$ of $\bt$ will be that of R-relations. Notice that for any set $X$ the collection $R(X)$ of all R-relations on $X$ forms a complete lattice where meets are simply intersections and thus $\br$ is fibred in complete lattices. For notational convenience we interchangeably use the pair $(X,[\cdot,\cdot,\cdot])$ and $[\cdot,\cdot,\cdot]$ to denote objects in $\bt$ and $\br$; in the latter the underlying set will always be known from context. For $i\leq 2$, define the functors $\li:\bt \to \bt$ 
\[
\lone([\cdot,\cdot,\cdot]) = [\cdot,\cdot,\cdot] \cup \{(a,b,b)  \mid a,b\in X\}
\]
\[
\ltwo([\cdot,\cdot,\cdot]) = [\cdot,\cdot,\cdot] \cup \{(a,b,c)  \mid (c,b,a) \in [\cdot,\cdot,\cdot]\}.
\]

\noindent
The following can be easily verified.

\begin{lemma} The functor $\ltwo$ preserves R1 and $\ltwo\circ \lone : \bt \to \ltwo\circ \lone[\bt]$ is a left adjoint.
\end{lemma}

The axiom R3 is set-theoretic in nature: it is the only axiom that mentions equality of elements. Consequently, R3 closures of ternary relations will inevitably modify the underlying set. To develop R3 closures of ternary relations, it is not enough to glue points together that the axiom R3 demands should be the same. The following example illustrates just that. 
\begin{example} Put $X = \{a_1,a_2,a_3,a_4\}$ and let $[\cdot,\cdot,\cdot]$ be the following:
\[
X_\bot \cup \{(a_1,a_2,a_1), (a_1,a_3,a_2), (a_2,a_3,a_1), (a_3,a_4,a_2), (a_2,a_4,a_3)\}.
\]

\noindent
Gluing $a_1$ and $a_2$ together, call this new point $y$, yields $Y = \{y, a_2,a_3,a_4\}$ where the smallest ternary relation on $Y$ making the quotient map $X \to Y$ monotone must be:

\[
[\cdot,\cdot,\cdot]_Y = Y_\bot \cup \{(y,a_3,y), (y,a_3,y), (a_3,a_4,y), (y,a_4,a_3)\}.
\]

\noindent
Again, this new ternary relation does not satisfy R3, and we must glue $y$ with $a_3$; denote this new point $z$. The set $Z = \{z, a_4\}$ must be endowed with the ternary relation

\[
[\cdot,\cdot,\cdot]_Z = Z_\bot \cup \{(z,a_4,z)\},
\]

\noindent
so as to make the quotient map $Y \to Z$ monotone. This leads us to finally glue $z$ and $a_4$ together and end up with a one-point set: the actual R3 closure of $(X,[
[\cdot,\cdot,\cdot])$. 
\end{example}
The above example can be easily extended to an infinite set endowed with a ternary relation that would require of an infinite number of gluing steps to achieve its R3 closure. In a nutshell, the act of gluing points and then endowing the resulting quotient set with the smallest ternary relation making the quotient map monotone can lead to a new ternary relation the fails to satisfy R3. It is this observation that leads us to the following construction.

 Take any $\bt$-object $(X,[\cdot,\cdot,\cdot])$ and for any $a,b \in X$ denote
\begin{itemize}
\item $a \sim_1 a$ if, and only if,  $a = b$.
\medskip
\item $a \sim_{n+1} b$ if, and only if, $a \sim_n b$ or $\exists m\in \N$ with $\{[ r_i, s_i,t_i ] \in [\cdot,\cdot,\cdot]\}_{i\leq m}$ so that 
\medskip
\begin{itemize}
\item  $a\in \{r_m,s_m,t_m\}$ and $b \in \{r_1,s_1,t_1\}$, and
\medskip
\item for all $i\leq m$, $r_i \sim_n t_{i}$ and $x  \sim_n y$ with $x\in\{r_i,s_i,t_i\}$ and $y\in\{r_{i+1},s_{i+1},t_{i+1}\}$.
\end{itemize}
\end{itemize}
 \noindent
 We demand $a \sim_\omega b$ precisely when $a\sim_n b$ for some $n \in \N$.
 
 \begin{lemma}\label{lem:equiv} The relation $\sim_\omega$ defined above is an equivalence relation on $X$.
 \end{lemma}

 \begin{proof} Reflexivity and symmetry are obvious so let $a \sim_\omega b$ and $b\sim_\omega c$. It follows that for some $k \in \N$ we have $a \sim_{k+1} b$ and $b \sim_{k+1} c$. By definition, there exist finite sequences $ \{[ r_i,s_i,t_i] \mid i\leq m\}$ and $ \{[ u_i,v_i,w_i] \mid i \leq n \}$ so that
 $$
  [r_1,s_1,t_1] \wedge \ldots \wedge  [r_m,s_m,t_m] \mbox{ and }
 $$
 $$
  [u_1,v_1,w_1]  \wedge \ldots \wedge  [u_n,v_n,w_n]
 $$
 \noindent
 with the defining relationships between the variables $r_i,s_i,t_i,u_i,v_i$ and $w_i$, and the points $a,b$ and $c$. Since $\{b\}\subseteq \{r_m,s_m,t_m\} \cap \{u_1,v_1,w_1\}$, the two chains above can be concatenated so as to form a chain from $a$ to $c$, in which case $a \sim_{k+1} c$ and transitivity holds. 
 \end{proof}

Next, put $Z = X/{\sim_\omega}$ with quotient map $p: X \to Z$ and ternary relation $[\cdot,\cdot,\cdot]_Z$ on $Z$ so that $[a,b,c]_Z \Leftrightarrow [x,y,z]$ for some $x\in a$, $y\in b$ and $z\in c$; by design, the quotient map $p: X \to Z$ is monotone.

 \begin{lemma}\label{lem:R3}  For $\lthree:\bt \to \lthree[\bt]$ by $L_3[(X,[\cdot,\cdot,\cdot])] = (Z,[\cdot,\cdot,\cdot]_Z) $, as illustrated above, we have:
 \begin{enumerate}[(a)]
\item $(Z,[\cdot,\cdot,\cdot]_Z)$ is an R3 relation and $L_3$ preserves R1 and R2, and
\medskip
\item the functor $\lthree:\bt \to \lthree[\bt]$ is left adjoint.
  \end{enumerate}
 
 \end{lemma}
\begin{proof} (a) We begin by showing that $(Z,[\cdot,\cdot,\cdot]_Z)$ satisfies R3. Let $[a,b,a]_Z$ for some $a,b \in Z$ and recall that this is true precisely when for some $x,z\in a$ and $y\in b$, $[x,y,z] $. Let $n\in \N$ be the smallest number for which $x\sim_n z$ (to avoid trivialities we assume $n>1$). It follows that there exists $\{[ r_i, s_i,t_i ] \}_{i\leq m}$ (for some $m\in \N$) with $x\in\{r_1,s_1,t_1\}$, $z\in\{r_m,s_m,t_m\}$ and so that for all $i\leq m$, $s_i \sim_n t_{i}$ and, $u  \sim_n v$ with $u\in\{r_i,s_i,t_i\}$ and $v\in\{r_{i+1},s_{i+1},t_{i+1}\}$. By adjoining $[x,y,z]$ to $\{[ r_i, s_i,t_i ] \}_{i\leq m}$ and since $x\sim_n z$ we have that $x\sim_{n+1} y \sim_{n+1} z$ and that $a=b$. That $L_3$ preserves R1 and R2 can be easily verfied.

(b) Assume that for some R$_3$-relation $(W,[\cdot,\cdot,\cdot]_W)$ there exists a monotone $f:X\to W$. We first show that for all $a\in X$ the assignment for which $p(a) \mapsto f(a)$ is well-defined ($p:X\to Z$ is the quotient map). Clearly, if $a\sim_1b$ for $a,b\in X$, then $ f(a) = f(b)$. Assume this holds up to some $k\in \N$ (i.e., $a\sim_k b \Rightarrow  f(a) = f(b) $).  If $a\sim_{k+1} b$ then we can find  $ \{[ r_i,s_i,t_i]\in [\cdot,\cdot,\cdot] \mid i\leq n\}$ with $a\in\{r_1,s_1,t_1\}$, $b\in\{r_m,s_m,t_m\}$ and so that for all $i\leq m$, $s_i \sim_n t_{i}$ and, $u  \sim_n v$ with $u\in\{r_i,s_i,t_i\}$ and $v\in\{r_{i+1},s_{i+1},t_{i+1}\}$. For all $i \leq m$, by the inductive hypothesis, $f(r_i)=f(s_i)=f(t_i)$ and $f(u) = f(v)$ for some $u\in\{r_i,s_i,t_i\}$ and some $v\in\{r_{i+1},s_{i+1},t_{i+1}\}$. It follows that $\forall n\in \N$, $a\sim_n b \Rightarrow f(b) = f(a)$ and that $a\sim_\omega b \Rightarrow f(a) = f(b)$. That is, $f$ maps elements from the same equivalence class in $\sim_\omega$ to the same element in $W$. Hence, the assignment for which $p(a) \mapsto f(a)$ is a well-defined function; denote it by $f'$ and notice that $f = f'\circ p$. Lastly, we show monotonicity of $f'$. For $a,b,c\in Z$, $[a,b,c]_Z$ occurs when for some $x\in a$, $y\in b$ and $z\in c$ we have $[x,y,z]$. That $f$ is monotone yields $[f(x),f(y),f(z)]_W$, in which case $[f'(a),f'(b),f'(c)]_W$, since $f = f'\circ p$. Hence the functor $\lthree:\bt \to \lthree[\bt]$ is left adjoint as claimed.
\end{proof}

The axiom R4 is largely more straightforward. Given any ternary relation $(X,[\cdot,\cdot,\cdot])$ and a pair $a,b \in X$ define recursively,

\begin{itemize}
\item $[a,b]_1 = [a,b]$
\medskip
\item $[a,b]_{n+1} = [a,b]_n \cup \{x\in X \mid \exists c,d \in [a,b]_n$ so that $x \in [c,d]_n\}$.
\end{itemize}
\noindent
Next, let $[a,b]_\omega = \bigcup_{n \in \omega} [a,b]_n$ and $(X,[\cdot,\cdot,\cdot]_\omega)$ be the one generated by the intervals $[a,b]_\omega$.

 \begin{lemma}\label{lem:R4}  For $\lfour:\bt \to \bt$ so that $(X,[\cdot,\cdot,\cdot]) \mapsto (X,[\cdot,\cdot,\cdot]_\omega) $, as illustrated above, we have:
 \begin{enumerate}[(a)]
\item $(X,[\cdot,\cdot,\cdot]_\omega)$ is an R4 relation and $L_4$ preserves R1-R3, and
\medskip
\item the functor $\lfour:\bt \to \lfour[\bt]$ is left adjoint.
  \end{enumerate}
\end{lemma}
\begin{proof} (a) First we show $(X,[\cdot,\cdot,\cdot]_\omega)$ is an R$_4$-relation: let $b,d\in[a,c]_\omega$ and $x\in[b,d]_\omega$ and, by design, let $n\in \N$ be any number so that $b,d\in[a,c]_n$ and $x\in[b,d]_n$. It follows that $x\in[a,c]_{n+1}$ and that $x\in[a,c]_\omega$. For (a), that $[\cdot,\cdot,\cdot]_\omega$ preserves R1 and R2 is clear and if $b\in [a,a]_1$ then $b=a$. Assume this is true up to $k\in \N$. If $b\in [a,a]_{k+1} $ then $b \in [a,a]_k$ (in which case we are done) or we can find $c,d \in [a,a]_k$ so that $b\in [c,d]_k$. Immediately we get that $a=c=b=d$. Thus, $[a,b,a]_\omega$ implies $b=a$ and that $[a,a]_\omega = \{a\}$ for all $a \in X$. Consequently, we have that $[\cdot,\cdot,\cdot]_\omega$ satisfies R3.  

Next, we show (b). Assume that for some R$_4$-relation $(W,[\cdot,\cdot,\cdot]_W)$ there exists a monotone $f:(X,[\cdot,\cdot,\cdot]) \to (W,,[\cdot,\cdot,\cdot]_W)$. We need only show that $f:(X,[\cdot,\cdot,\cdot]_\omega) \to (W,[\cdot,\cdot,\cdot]_W)$ is monotone also. For any $a,b,c\in X$, that $[a,b,c]_1$ implies $[a,b,c]$ and thus $[f(a),f(b),f(c)]_W$. Assume this to hold up to some $k\in \N$ and let $[a,b,c]_{k+1}$. Obviously, for $[a,b,c]_k$ we have our result. Otherwise, there exist $d,e \in [a,c]_k$ with $b \in [d,e]_k$. By the inductive hypothesis, $f(d),f(e) \in [f(a),f(c)]_W$ and $f(b) \in [f(d),f(e)]_W$ and since $[\cdot,\cdot,\cdot]_W$ satisfies R4 then $f(b) \in [f(a),f(c)]_W$. 
\end{proof}
In what follows we denote $L = L_4\circ L_3\circ L_2\circ L_1$.

\begin{corollary}\label{cor:joinsexplicit} The category $\br$ is a reflective subcategory of $\bt$ with reflector
 $L:\bt \to L[\bt]=\br$. In particular, for any collection $\{[\cdot,\cdot,\cdot]_i\mid i \in I\}$ of R-relations on a set $X$, $L$ gives an explicit form for their join in $R(X)$:
 \[
 \bigvee [\cdot,\cdot,\cdot]_i = L\left[\bigcup [\cdot,\cdot,\cdot]_i \right].
 \]
\end{corollary}

In light of the above, notice that for any R-relation $(X,[\cdot,\cdot,\cdot]_X)$, a set $Y$ and a function $Y \to X$ there exsists a $U$-initial lift $(Y',[\cdot,\cdot,\cdot])$, where $U:\bt \to \bs$ denotes the forgetful functor and $Y'$ is the quotient of $Y$ dictated by $L$. Indeed, this is a simple consequence of the compositions $\br \leftrightarrows \bt \leftrightarrows \bs$ and that $\bt$ is topological over $\bs$. More to the point, that $U: \br \to \bs$ is fibred in lattices suggests the possibility of $\br$ being topological over $\bs$.

\begin{lemma}\label{lem:cones} Any cone $\left(Y \to U[(X_j,[\cdot,\cdot,\cdot]^j]\right)_{j\in J}$ has a $U$-initial lift and, thus, $\br$ is topological over $\bs$.
\end{lemma}
\begin{proof}
This is a direct consequence of the previous corollary. For all $j\in J$ let 
\[
[\cdot,\cdot,\cdot]_j = \bigvee \{ [\cdot,\cdot,\cdot]\in R(Y) \mid  (Y,[\cdot,\cdot,\cdot]) \to (X_j,[\cdot,\cdot,\cdot]^j) \text{ is monotone}\}.
\]
\noindent
It follows that for any fixed $i\in J$, $(Y, [\cdot,\cdot,\cdot]_i) \to (X_i,[\cdot,\cdot,\cdot]^i)$ is monotone. So as to satisfy all $Y \to U[(X_j,[\cdot,\cdot,\cdot]^j$, we take a meet in $R(Y)$:
\[
[\cdot,\cdot,\cdot] = \bigwedge_{j\in J} [\cdot,\cdot,\cdot]_j.
\] 
\noindent
Again, $(Y,[\cdot,\cdot,\cdot]) \to (X_j,[\cdot,\cdot,\cdot]^j)$ is monotone for all $j\in J$. The remaining universal properties of $U$-initial lifts can be easily verified for $(Y,[\cdot,\cdot,\cdot])$.
\end{proof}

Let ${\bf CMLat}$ denote the collection of complete lattices with arbitrary  meet  preserving  and  nonempty  join preserving  functions. 

\begin{corollary} The assignment $X \to R(X)$ with $f:X\to Y \mapsto f_*:R(Y) \to R(X)$ where for $[\cdot,\cdot,\cdot] \in R(Y)$
\[
f_*([\cdot,\cdot,\cdot]) = \bigvee\{[\cdot,\cdot,\cdot]'\in R(X) \mid f: [\cdot,\cdot,\cdot]' \to [\cdot,\cdot,\cdot] \text{ is monotone}\},
\]

\noindent
defines a contravariant functor $\bs \to {\bf CMLat}$ and a Grothendieck fibration on $\br \to \bs$.
\end{corollary}

\begin{proof} Denote this assignment by $F_*$. By design, one can easily verify that with the exception of empty joins (consider any constant function for a counterexample here), any $f_*$ preserves all other joins and meets and $F_*$ is well defined. Next we show functoriality of $F_*$: Lemma~\ref{lem:cones} demonstrates that for any $[\cdot,\cdot,\cdot] \in R(Y)$ and $f:X\to Y$ then R-relation $f_*([\cdot,\cdot,\cdot])\in R(X)$ is the finest making $f$ monotone. Moreover, given $h:X \to Z = (g:Y\to Z)\circ (f:X\to Y)$ and $[\cdot,\cdot,\cdot] \in R(Z)$ then Corollary~\ref{cor:joinsexplicit} yields $(a,b,c) \in h_*([\cdot,\cdot,\cdot])$ precisely when $(h(a), h(b), h(c)) \in [\cdot,\cdot,\cdot]$. Hence,
\begin{align*}
(a,b,c) \in f_*\circ g_* ([\cdot,\cdot,\cdot]) &\Longleftrightarrow (a,b,c) \in f_*\left( g_* ([\cdot,\cdot,\cdot])\right)\\
						               &\Longleftrightarrow   (f(a),f(b),f(c)) \in g_* ([\cdot,\cdot,\cdot])\\
						               &\Longleftrightarrow ((g \circ f)(a),(g \circ f)(b),(g \circ f)(c)) \in[\cdot,\cdot,\cdot]\\
						               &\Longleftrightarrow (h(a), h(b), h(c)) \in [\cdot,\cdot,\cdot].
\end{align*}
Thus, $F_*$ is functorial. Lastly, one can easily verify that the forgetful functor $\br \to \bs$ is cartesian, and the proof is complete.
\end{proof}

\begin{remark} The functor $\br \to {\bf Set}$ is not a opfibration. This fact can be appreciated by noting that the axiom R3 is distinct from the other axioms in the sense that unlike the latter, R3 closures of objects in $\bt$ inevitably modify the underlying set (not just the relation). Consequently, in general, the dual assignment to $F_*$ is not functorial.
\end{remark}

\subsection{Fra\"{i}ss\'{e} property} The question whether a given class of finite objects is a Fra\"{i}ss\'{e} class is connected to the existence of universal homogeneous objects for related classes via Fra\"{i}ss\'{e}'s amalgamation theorem. While the theory was developed in the 1950's there has been a great deal of work recently on developing further connections between Fra\"{i}ss\'{e}'s theory of finite structures and their amalgamation properties with Ramsey theory and amenability of the automorphism groups of the countable homogeneous structures. For a categorical view of Fra\"{i}ss\'{e} structures, see \cite{MR2140630} and see the paper \cite{MR3244668} for a systematic study of the connections between Fra\"{i}ss\'{e} theory, topological dynamics and Ramsey theory. 

Recall that a countable relational structure $M$ is said to be {\bf homogeneous}, if for every
isomorphism $ U \to V$ between finite substructures $U$ and $V$ of $M$, there is an automorphism of $M$ extending $ U \to V$. For example, in the category of linear orders and order preserving maps, $(\Q,<)$ is one such structure. Letting Age($M$) represent the collection of all finite substructures of $M$, Fraiss\'{e}'s Theorem \cite{MR0057220} proves that homogeneity of $M$ implies that Age($M$):
\begin{enumerate}[(a)]
\item is closed under isomorphism and substructure;
\smallskip
\item has countably many members up to isomorphism;
\smallskip
\item has the {\bf joint embedding property} (JEP): if $U, V \in$ Age($M$) then there is $W \in$ Age($M$) such that both $U$ and $V$
embed in $W$; and
\smallskip
\item has the {\bf amalgamation property} (AP): whenever $A, B_1, B_2 \in C$ and
$f_i : A \to B_i$ (for $i = 1, 2$) are embeddings, there is $C\in$ Age($M$) and embeddings $g_i
: B_i \to C$ (for $i = 1, 2$) such that $g_1 \circ f_1 = g_2 \circ f_2$.
\end{enumerate}

Conversely, Fra\"{i}ss\'{e} also showed that given a class ${\mathcal C}$ of finite structures over a finite language $L$, if ${\mathcal C}$ is closed under isomorphisms and substructures, has countably many members up to isomorphism, and has the (JEP) and (AP), then there is a unique (up to isomorphism) countable homogeneous structure $M$ such that Age$(M) = {\mathcal C}$. Such a class ${\mathcal C}$ is called a Fra\"{i}ss\'{e} class and the resulting homogeneous object the Fra\"{i}ss\'{e} limit of the class. And indeed, $({\mathbb Q},\leq)$ is precisely the Fra\"{i}ss\'{e} limit of the class of finite linear orders. Two other important examples of such classes (and their limits) are the class of finite metric space, whose limit is the Urysohn space, and the class of finite graphs, whose limit is the Rado graph. 

Our main observation here is that Lemma \ref{lem:cones} implies that the class of finite R-structures has both the (JEP) and the (AP). And since this class is obviously closed under isomorphism, substructure, and has countably many isomorphism classes, we obtain:

\begin{theorem} The class of finite R-structures is a Fra\"{i}ss\'{e} class, hence has a Fra\"{i}ss\'{e} limit. 
\end{theorem}

\section{Antisymmetry and complete distributivity}\label{sec:sepanddist}

 An {\bf antisymmetric} R-relation is one for which $[a,b,c]$ and $[a,c,b]$ implies $b=c$. Antisymmetric R-relations abound in mathematics and can be found in areas ranging from lattice theory to continuum topology  (see \cite{pre06252969}, \cite{MR1383662} and \cite{MR1120620}). This section is concerned with studying these relations from a lattice-theoretic perspective and exposing them as complete and distributive closures of complete lattices. We first present antisymmetric closures of R-relations and in passing address an issue raised in \cite{pre06252969}, where the author asks whether or not these closures exist. Order-theoretically, in the following section we show it represents a reflector between complete lattices and complete and distributive counterparts.

\subsection{Antisymmetry.}

Let $\ba$ denote the full subcategory of $\br$ whose objects are antisymmetric R-relations. The reader can verify
that $\ba$ is complete and contains all coproducts and all such construction
agree with those in $\br$. It also has all coequalizers, but those are not as
simple to illustrate and will be constructed in the sequel. Interpreted within a categorical setting, in \cite{pre06252969} (pg. 15) the
author proves no right adjoint $\ba \hookrightarrow \br$ can exist. In light of this, it is
natural to question the existence of a left adjoint and this is precisely what
we demonstrate to exist, in what follows. Consider any ternary relation
$(X,[\cdot,\cdot,\cdot])$. We define the following equivalence relation $\sim$ on $X$ recursively as follows:

\begin{itemize}
\item $a\sim_0b \Leftrightarrow a=b$
\medskip
\item $a\sim_{r+1} b \Leftrightarrow a\sim_r b$ or $ \exists \{[m_i, x_i\,y_i]
\mid i\leq n \in \N\}$ for which
\medskip
\begin{itemize}
\item $a\sim_r p \in \{x_0,y_0\}$ and $b\sim_r p \in \{x_n,y_n\}$, 
\medskip
\item $m_i \sim_r m_{i+1}$, $x_i \sim_r x_{i+1}$ and $y_i \sim_r y_{i+1}$  for
$i$ odd, and
\medskip
\item for some $p \in \{y_i, x_i\}$ and $q\in \{y_{i+1},x_{i+1}\}$ we have
$p\sim_r q$ for $i$ even.
\end{itemize}
\end{itemize}
\noindent
We demand for $a\sim b$ provided $a\sim_n b$ for some $n$ and put
$X/{\sim} = Y$; it's simple to verify that $\sim_n$ is an
equivalence relation on $X$ for each $n\in\N$ and thus $\sim$ is also.
This process only defines the underlying set of the antisymmetrised relation
while the ternary relation is the canonical one: $[a,b,c]_Y$ provided that
$[x,y,z]$ for some (any) $x\in a$, $y\in b$ and $z\in c$. The ternary relation $(Y,
[\cdot,\cdot,\cdot]_Y)$ is antisymmetric: for if $[a,b,c]_Y, [a,c,b]_Y$ then we can find $x_1,x_2 \in a$, $y_1,y_2\in b$ and
$z_1,z_2 \in c$ so that $[x_1,y_1,z_1]$ and $[x_2,z_2,y_2]$. Denote $L_A$ the
operator for the above defined process. It clearly preserves R1, R2, and it also preserves R3 when
R1 is involved. It might not preserve R4 and, as the following example
illustrates, $L_4$ might not preserve antisymmetry either. 

\begin{example} Begin with a set $X= \{ a,b,c,x,d_1,d_2,y\}$ and 
$$[\cdot,\cdot,\cdot] = X_\bot \cup \{[a,b,c], [a,d_1,c],[b,x,d_2],[y,d_1,d_2],[y,d_2,d_1],[a,c,d_2],[a,c,x]\}.$$
 Notice that the point $y$ and its relation to $d_1$ and $d_2$ makes $[\cdot,\cdot,\cdot]$ not antisymmetric. Applying $L_A$ glues $d_1$ and $d_2$ together to a point $d$ and generates 
 \[[\cdot,\cdot,\cdot]_1 = Y_\bot \cup \{[a,b,c], [a,d,c],[b,x,d],[a,c,x]\}
 \]
 
 \noindent
  where $Y= \{ a,b,c,x,d,y\}.$
  This is not an R4 relation since $[a,x,c]$ is not present. Applying $L_4$ then leaves the underlying set unchanged but modifies the ternary relation to 
 \[
 [\cdot,\cdot,\cdot]_2 = Y_\bot \cup \{[a,b,c], [a,d,c],[b,x,d],[a,c,x], [a,x,c]\}
 \]
 \noindent
   by adding $[a,x,c]$. Again, we witness a non-antisymmetric R-relation because $[a,c,x], [a,x,c]$ are present. Applying $L_A$ once again then finishes the process by gluing $x$ to $c$ leaving an antisymmetric R-relation. 

The above scenario can be extended so that the composition $L_A\circ L_4\circ L_A$, as previously described, will not yield an antisymmetric R-relation. Indeed, to $(X,[\cdot,\cdot,\cdot])$ as initially described, consider adjoining the antisymmetric R-relation $X' = \{ a',b',c',x',d'_1,d'_2\}$ with relation 
\[
[\cdot,\cdot,\cdot]' = X'_\bot \cup \{[a',b',c'], [a',d'_1,c'],[b',x',d'_2],[a',c',d'_2],[a',c',x']\}
\]
\noindent
as follows: let $Z = X\cup X'$ with $[\cdot,\cdot,\cdot]_Z = Z_\bot \cup [\cdot,\cdot,\cdot] \cup [\cdot,\cdot,\cdot]' \cup \{[x,d'_1,d'_2],[c,d'_2,d'_1]\}$. Applying  $L_A\circ L_4\circ L_A$ to $(Z,[\cdot,\cdot,\cdot]_Z)$ only modifies $X$ and its ternary relation, since $(X',[\cdot,\cdot,\cdot]')$ was designed as an antisymmetric R-relation and the relationship between points from $X$ and $X'$ violate neither R4 nor antisymmetry. However, the last $L_A$ glues $x$ and $c$ together to a point $y'$ playing the same role $y$ did with $d_1$ and $d_2$ but with $d'_1$ and $d'_2$ instead, thus we find ourselves in much the same scenario as in the beginning of the example necessitating several uses of $L_4$ and $L_A$ to yield an antisymmetric R-relation. The reader can then appreciate how, carefully, appending an infinite countable number of copies of $(X',[\cdot,\cdot,\cdot]')$ can yield a scenario where no finite use of $L_A$ after $L_4$ will ever yield antisymmetric R-relation. \hspace{1.6 in} $\square$

\end{example}

The previous example is interesting in that it suggests a
direct limit of the functors $L_A$ and $L_4$. That is, what we show defines the
left adjoint of $\ba \hookrightarrow \bt$ is actually the direct limit of
applying $L_4$ after $L_A$, $\omega$-times. For any R-relation
$(X,[\cdot,\cdot,\cdot])$ let $(X_1,[\cdot,\cdot,\cdot]_A^1) =
L_A[(X,[\cdot,\cdot,\cdot])]$, $(X_1,[\cdot,\cdot,\cdot]_1) =
L_4[(X_1,[\cdot,\cdot,\cdot]_A^1)]$, and in general 
\[(
X_n,[\cdot,\cdot,\cdot]_n^A) = L_A[(X_{n-1},[\cdot,\cdot,\cdot]_{n-1})] \mbox{ and }
\]
 \[
 (X_n,[\cdot,\cdot,\cdot]_n) =
L_4[(X_{n-1},[\cdot,\cdot,\cdot]_A^{n-1})].
\]
 By design, we have: 
 \medskip
 
\xymatrix{ 
(X,[\cdot,\cdot,\cdot]) \ar@{.>}[rd]\ar[r]^{q_1} & (X_1,[\cdot,\cdot,\cdot]_A^1) \ar[d] & & &\\
& (X_1,[\cdot,\cdot,\cdot]_1)\ar[r]^{q_2} \ar@{.>}[rd]& \ldots \ar[d] & & \\
& & (X_{n-1},[\cdot,\cdot,\cdot]_{n-1}) \ar[r]^{q_n} \ar@{.>}[rd]& (X_n,[\cdot,\cdot,\cdot]_A^n) \ar[d] & \\
 & & & \ldots
}

\medskip

\noindent
where the maps are identities and quotients, and thus monotone. Next, for any
$a,b\in X$, let $a\sim_A b$ provided that for some $n\in \N$, the quotient map
$q_n: X \to X_{n}$ glues the points together; $q_n(a) = q_n(b)$. Obviously,
$\sim_A$ is an equivalence relation so let $X_A = X/{\sim_A}$. Lastly, let
$[\cdot,\cdot,\cdot]_A$ so that $(a,b,c) \in [\cdot,\cdot,\cdot]_A$ with: for
$x\in a$, $y\in b$ and $z\in c$ we can find an some $n\in \N$ so that
$[q_n(x),q_n(y),q_n(z)]^n$.  We seek to show that for any antisymmetric
R-relation $(Y,[\cdot,\cdot,\cdot]_Y)$ with monotone $(X_n,
[\cdot,\cdot,\cdot]_n) \to (Y,[\cdot,\cdot,\cdot]_Y)$ so that for all $n$

\[
\xymatrix{
\ldots \ar[r] & (X_{n-1},[\cdot,\cdot,\cdot]_{n-1}) \ar[rr]^{q_n}\ar[ddr] & &
(X_n,[\cdot,\cdot,\cdot]_A^n)\ar[r]\ar[ddl] &\ldots \\
&  \\
& & (Y,[\cdot,\cdot,\cdot]_Y)\\
}
\]
\noindent
commutes, then we can find a unique $(X_A,[\cdot,\cdot,\cdot]_A) \to
(Y,[\cdot,\cdot,\cdot]_Y)$ so that 

\[
\xymatrix{
\ldots \ar[r] & (X_{n-1},[\cdot,\cdot,\cdot]_{n-1}) \ar[rr]^{q_n}\ar@/_2.5pc/[ddr]\ar@/_1.5pc/[dr]  & &
(X_n,[\cdot,\cdot,\cdot]_A^n)\ar[r]\ar@/^2.5pc/[ddl] \ar@/^1.5pc/[dl] &\ldots \\
& & (X_A,[\cdot,\cdot,\cdot]_A)\ar[d]\\
& &(Y,[\cdot,\cdot,\cdot]_Y)\\
}
\]
\noindent
does as well. 

\begin{lemma} For any R-relation $(X,[\cdot,\cdot,\cdot])$,
$(X_A,[\cdot,\cdot,\cdot]_A)$ is an antisymmetric R-relation for which

\begin{enumerate}
\item The quotient map $q:X \to X_A$ is monotone, and
\medskip
\item Given any other antisymmetric $(Y,[\cdot,\cdot,\cdot]_Y)$ with monotone
$f: X \to Y$ there exists a unique monotone $g: X_A \to Y$ so that the diagram

\[
\xymatrix{
(X,[\cdot,\cdot,\cdot])\ar[rr]^q\ar[ddr]_f &
&(X_A,[\cdot,\cdot,\cdot]_A)\ar[ddl]^g\\
&  \\
& (Y,[\cdot,\cdot,\cdot]_Y)
}
\]
commutes.
\end{enumerate}
\end{lemma}
\begin{proof} That $[\cdot,\cdot,\cdot]_A$ satisfies R1 and R2 is obvious. Next
we show antisymmetry: let $[a,b,c]_A$ and $[a,c,b]_A$. This means that for some
$r,s\in \N$, with $x_1,x_2\in a$, $y_1,y_2\in b$ and $z_1,z_2\in c$ we have
$[q_r(x_1),q_r(y_1),q_r(z_1)]_A^r$ and $[q_s(x_2),q_s(z_2),q_s(y_2)]_A^s$. Without
loss of generality, let $s\geq r$ and notice that
$[q_s(x_1),q_s(y_1),q_s(z_1)]^s$ and thus claim follows. For R3 we need only
notice that R1 $+$ antisymmetry $\Rightarrow$ R3. That R4 holds is proved in much
the same way. 
\noindent
By design, $(1)$ should be clear. For universality, and to prove $(2)$, we first
show that $g$ is always well-defined. Let $a\in X_A$ and $x,y\in a$. Assume that
$x\sim_1 y$; that is, there exists a $z\in X$ so that $[z,x,y]$ and $[z,y,x]$.
Since $[\cdot,\cdot,\cdot]_Y$ is antisymmetric, then $f(x) = f(y)$. Assume this
holds up to some $k\in \N$ and that $x\sim_{k+1}y$. Again, since
$[\cdot,\cdot,\cdot]_Y$ is antisymmetric, it must be that $f(x)=f(y)$. Finally,
to prove that $g$ is monotone, recall that $L_4$ always yields a universal
R4-relation. With this in mind, let $[a,b,c]_A$. Immediately we obtain that for
some $n\in \N$, $[q_n(x),q_n(y),q_n(z)]^n$ with $x\in a$, $y\in b$, and $z\in
c$. By universality of $L_4$, then it must be that $[f(x),f(y),f(z)]_Y$ and the
proof is complete.
\end{proof}

\begin{theorem} The bicomplete category $\ba$ is a reflective, but not coreflective,
subcategory of $\br$.
\end{theorem}
\begin{proof} Coequalisers were the only constructions left to produce and those
follow directly from the previous lemma.
To prove that $\ba \hookrightarrow \bt$ is not left adjoint consider the set $X
= \{a,b,c\}$ with $[\cdot,\cdot,\cdot] = X_\bot \cup \{[a,b,c],[a,c,b]\}$ (a
non-antisymmetric R-relation). To show that no universal arrow
$((Z,[\cdot,\cdot,\cdot]_Z),Z\to X)$ exists first notice that for $i \in X$ and $\overline{1} = (\{1\}, \{(1,1,1)\}$ the functions $f_i
: \overline{1} \to X$ for which $1 \mapsto i \in X$ demand for exactly three
distinct elements in $Z$ to be between themselves; for $a^*,b^*,c^* \in Z$,
$[a^*,a^*,a^*]_Z$, $[b^*,b^*,b^*]_Z$ and $[c^*,c^*,c^*]_Z$. Moreover, the arrow $Z\to X$
must map $a^* \mapsto a$, $b^*\mapsto b$ and $c^* \mapsto c$. Next, take $(X,
[\cdot,\cdot,\cdot]\smallsetminus \{[a,b,c]\})$ and notice that the
corresponding arrow for the identity on $X$ yields that $[a^*,c^*,b^*]_Z$. Applying
the very same logic to $(X, [\cdot,\cdot,\cdot]\smallsetminus \{[a,c,b]\})$
yields $[a^*,b^*,c^*]_Z$ and since $b^*\not = c'$ the proof is complete.
\end{proof}

\subsection{Distributive closures of complete lattices.} 
In \cite{pre06252969} the notion of a
road system is exemplified for lattices as
follows: to a lattice $L$ one can associate a road system by letting roads be the order-theoretic intervals $[a\wedge b, a\vee b]$ for all $a,b \in L$. That is, the betweenness intervals become
$[a,b] = \{ y \mid a\wedge b \leq y \leq a\vee b\}$. Characterisations of certain ordered sets via betweenness relations they
generate abound in the literature. Some betweenness relations are considered to
occur only when elements are present in a linear suborder \cite{ResultsChara}.
That is, $[a,b,c]$ provided $a < b < c$ or its converse. Others, as in
\cite{MR1120620} and \cite{MR1383662}, extend this notion to all partial orders by considering $[a,b,c]$
precisely when $(a\vee b)\wedge (b\vee c) = b = (a \wedge b) \vee (b\wedge c)$; a much stronger condition than the one employed in this paper. Before establishing the following result we introduce an additional betweenness axiom closely related to R-relations generated from the Q, K and C-interpretations of betweenness in continua (\cite{MR3101640} and \cite{MR3268066}). An R-relation $(X,[\cdot,\cdot,\cdot])$ is {\bf disjunctive} provided that for all $a,b,c,x \in X$ if $x\in[a,b]$ then $[a,x,c]$ or $[c,x,b]$.

\begin{lemma}\label{lem:latticetorel} Let  $(X, [\cdot,\cdot,\cdot])$  be an R-relation generated from the intervals
of a lattice $(X,\leq)$, then $(X,\leq)$ is linear if and only if $(X, [\cdot,\cdot,\cdot])$ is disjunctive. Also, $(X,\leq)$ is bounded precisely when there
exists a pair $\alpha,\beta \in X$ with
$[\alpha,\beta] = X$ and so that: $\forall x,y\in X$, $[x,\beta] = [y,\beta] \Rightarrow
x=y$, and for all finite $S\subseteq X$ we can find $p,q\in X$ with $\bigcap_{x\in S} [x,\beta] =
[p,\beta]$ and $\bigcap_{x\in S} [x,\alpha] =
[q,\alpha]$. In the following, all lattices are bounded.
\medskip
\begin{enumerate}[(a)]

\item $(X,\leq)$ is a complete lattice precisely when for all $S\subseteq X$ we can find a $p\in X$ with $\bigcap_{x\in S} [x,\beta] =
[p,\beta]$.
\medskip
\item $(X,\leq)$ is a modular lattice precisely when for all $x,y,c \in X$ if $[x,\beta] \subseteq [y,\beta]$ and $[x,c] = [y,c]$
then $x=y$.
\medskip
\item $(X,\leq)$ is a distributive lattice precisely when
$(X, [\cdot,\cdot,\cdot])$ is antisymmetric.
\medskip
\item $(X,\leq)$ is a complete and distributive lattice precisely when in addition to (a),
$(X, [\cdot,\cdot,\cdot])$ is antisymmetric.
\medskip
\item  $(X,\leq)$ is a Boolean lattice precisely when in addition to (c) for any
$x\in X$ we can find $y\in X$ so that $[x,y] = X$. 

\end{enumerate}

\end{lemma}

\begin{proof}  Disjunctivity clearly yields a linear order (this is also found in \cite{pre06252969}, pg. 396 Example 4.0.2 (i)). Given an R-relation $(X, [\cdot,\cdot,\cdot])$ generated by a bounded lattice $(X,\leq)$ it is possible to recover a bounded lattice as follows. Let $\alpha,\beta \in X$ as proposed in the premise.
Next, define $\leq'$ so that for any pair $x,y\in X$: $x\leq' y$ if, and only if, $[y,\beta] \subseteq [x,\beta]$. The meet and joins of a pair $x,y$ are then the elements $q$ and $p$ respectively. The converse is trivial.

Item (a) is straightforward: given $S \subseteq X$ the existence of $\bigvee S = p \in X$ is equivalent to the equality $\bigcap_{x\in S} [x,\beta] =
[p,\beta]$. The purpose of (b) is to avoid sublattices isomorphic to $N_5$. For (c), let $(X, [\cdot,\cdot,\cdot])$ be antisymmetric and assume that
$\leq$ is not distributive. In which case we can find sublattices of $(X,\leq)$
isomorphic to $N_5$ or $M_3$. Both scenarios yield intervals with three
elements (i.e., the middle ones) that would contradict $(X, [\cdot,\cdot,\cdot])$ being separative. The
converse is proved in much the same way. Items (d) and (e) are trivial.
\end{proof}

\begin{remark} 
Notice that given a Boolean lattice $(X,\leq)$ much can be lost when generating its betweenness relation $(X, [\cdot,\cdot,\cdot])$. Consider the case $(X,\leq) = (\PP(Y), \subseteq)$. Within the R-relation that $(X,\leq)$ generates, any pair of complemented elements (i.e., $A,B \in X$ with $A = Y\smallsetminus B$) are suitable candidates for the top and bottom elements of the reconstructed lattice from $(X, [\cdot,\cdot,\cdot])$.
\end{remark}

Let $(X, [\cdot,\cdot,\cdot])$ be an R-relation generated from a complete lattice $(X,\leq)$ and let $(Y, [\cdot,\cdot,\cdot]') = L_A[(X, [\cdot,\cdot,\cdot])]$ with its corresponding lattice $(Y,\leq')$. Clearly, the quotient map $X \to Y$ is a complete homomorphism. Let {\bf CLat}  denote the category of complete lattices with complete lattice morphisms and {\bf DCLat} its full subcategory of complete and distributive lattices. Next we reprove, from a betweenness perspective, a classical result in lattice theory.

\begin{corollary} The antisymmetric closure operator is a reflector between {\bf CLat} and {\bf DCLat}.
 
\end{corollary}

\begin{proof}
We need only show that one can extract a complete and distributive lattice from the separative closure of a complete lattice since the quotient map is a complete lattice morphism. Take a pair $a,b \in Y$ and observe that $[a,\top]\cap [b,\top]$ is a saturated interval in $(X,\leq)$ and thus the equivalence class of the smallest element, say $p$, is the one for which $[[p],\top] = [a,\top]\cap [b,\top]$. The very same can be achieved when replacing the pair $a,b$ with any arbitrary collection of elements from $Y$. Distributivity comes from antisymmetry and the proof is complete.
\end{proof}

\subsection{Dedekind-MacNeille completions.}

We finish the paper with a short observation on preservation of betweenness by Dedekind-MacNeille completions. These types of completions have been extensively studied in the past and the reader can find a categorical interpretation of the Dedekind-MacNeille in \cite{MR1137908}.

 Denote the category of bounded partial orders and monotone functions as {\bf PO}. Consider an object $(X,\leq)$ in {\bf PO} and its completion $(Y,\leq')$ with embedding $q:X \to Y$. Recall that $q$ is always order preserving, so consider a pair of incomparable $x,y\in X$ and $R(x,y)$ generated from the convex sets of $(X,\leq)$. By the very nature of cuts $q(R(x,y)) \subseteq [q(x)\wedge q(y),q(x)\vee q(y)]$ and betweenness is preserved. Unlike distributive closures of complete lattices, the Dedekind-MacNeille completion is not a reflector between {\bf PO} and {\bf CLat}. This failure can be blamed on the over-abundance of betweenness morphisms when compared to complete lattice morphisms (the latter always preserve betweenness, whereas betweenness preservation does not always preserve order). However, we note that {\bf PO}'s injective objects are those from {\bf CLat} and as a consequence, {\bf PO} has enough injectives. 

\section*{Acknowledgements}

The authors would like to extend their gratitude to Walter Tholen and Lili Shen for their many helpful suggestions to the paper. We would also like to extend our sincere gratitude to the referee for her/his very delicate review of our manuscript: we consider her/his comments as invaluable to our work and its exposition.

\bibliographystyle{plain}
\bibliography{mybib}

\end{document}